\documentclass[graybox]{svmult}

\usepackage{helvet}         
\usepackage{courier}        
\usepackage{type1cm}        
\usepackage{makeidx}         
\usepackage{graphicx}        
\usepackage{multicol}        
\usepackage[bottom]{footmisc}
\usepackage{bm}

\makeindex             

\usepackage{latexsym}
\usepackage{dsfont}
\usepackage{bbm}
\usepackage{amssymb}
\usepackage{amsmath}
\usepackage{color}
\usepackage{soul}

\def\cD{\mathcal{D}} 
 
\def\cF{\mathcal{F}}

\def\cL{\mathcal{L}}

\def\cN{\mathcal{N}}

\def\cS{\mathcal{S}}

\def\cW{\mathcal{W}}

\def\fA{\mathfrak{A}}

\def\Erw{\mathbb{E}}

\def\N{\mathbb{N}}
  
\def\Prob{\mathbb{P}} 

\def\R{\mathbb{R}}

\def\Var{\mathbb{V}{\rm ar\,}}

\def\eps{\varepsilon}

\def\vth{\vartheta}

\def\1{\vec{1}}
\def\3{{\ss}}

\def\eqdist{\stackrel{d}{=}}
\def\iprob{\stackrel{\Prob}{\to}}
\def\idist{\stackrel{d}{\to}}

\def\wh{\widehat}

\def\sign{\textsl{sign}}

\def\AsA{\textsl{(Add)}}
\def\AsM{\textsl{(Mult)}}

\allowdisplaybreaks[4] 

\begin{document}

\title*{Renewal approximation for the absorption time of a decreasing Markov chain}
\titlerunning{Absorption time of a decreasing Markov chain}
\author{Gerold Alsmeyer$^{1}$ and Alexander Marynych$^{1,2}$}
\institute{$^{1}$ Inst.~Math.~Statistics, Department
of Mathematics and Computer Science, University of M\"unster,
Einsteinstrasse 62, D-48149 M\"unster, Germany.\at
$^{2}$ Faculty of Cybernetics, Taras Shevchenko National University of Kyiv, 01601 Kyiv, Ukraine\at
\email{gerolda@math.uni-muenster.de, marynych@unicyb.kiev.ua}}

\maketitle

\abstract{We consider a Markov chain $(M_{n})_{n\ge 0}$ on the set $\N_{0}$ of nonnegative integers which is eventually decreasing, i.e. $\Prob\{M_{n+1}<M_{n}|M_{n}\ge a\}=1$ for some $a\in\N$ and all $n\ge 0$. We are interested in the asymptotic behaviour of the law of the stopping time $T=T(a):=\inf\{k\in\N_{0}: M_{k}<a\}$ under $\Prob_{n}:=\Prob(\cdot|M_{0}=n)$ as $n\to\infty$. Assuming that the decrements of $(M_{n})_{n\ge 0}$ given $M_{0}=n$ possess a kind of stationarity for large $n$, we derive sufficient conditions for the convergence in minimal $L^{p}$-distance of $\Prob_{n}((T-a_{n})/b_{n}\in\cdot)$ to some non-degenerate, proper law and give an explicit form of the constants $a_{n}$ and $b_{n}$.}

\bigskip

{\noindent \textbf{AMS 2000 subject classifications:} primary 60F05; secondary 60J10
}

{\noindent \textbf{Keywords:} Markov chain, absorption time, minimal $L^{p}$-distance, random recursion}, renewal theory.

\section{Introduction}\label{sec:intro}

The purpose of this paper is to study the asymptotic behavior of a class of integer-valued stochastic sequences $(T_{n})_{n\ge 0}$ that satisfy a random recursive equation of the form \eqref{main_rec} stated below. In many, though not all applications, such $T_{n}$ arise as absorption times for certain decreasing Markov chains.

To be more specific, let $(M_{n})_{n\ge 0}$ be a temporally homogeneous Markov chain on $\N_{0}$ with absorbing state 0 and transition matrix $P=(p_{i,j})_{i,j\ge 0}$ such that $p_{i,j}=0$ for $1\le i<j$. The last condition means that the chain is decreasing in the sense that
$$ \Prob(M_{n+1}<M_{n}|M_{n}\ge 1)\ =\ 1 $$
for all $n\ge 0$. Given this property, the time until absorption at 0, viz.
$$ T\ :=\ \inf\{k\ge 0:M_{k}=0\}, $$
is clearly an a.s. finite stopping time under each $\Prob_{n}:=\Prob(\cdot|M_{0}=n)$. Our purpose is to study the distribution of $T$ under $\Prob_{n}$ as $n\to\infty$. Obviously, for this goal the situation stated in the abstract, which is more general and also more commonly encountered in applications, can always be cast into the present framework by relabeling the states $i\ge a$ and collapsing all states $<a$ into one absorbing state.

Our analysis embarks on the simple observation that
\begin{equation}\label{eq:embarking recursion}
\Prob_{n}(T\in\cdot)\ =\ \sum_{k=1}^{n}\Prob_{n}(M_{1}=n-k,T\in\cdot)\ =\ \sum_{k=1}^{n}p_{n,n-k}\,\Prob_{k}(T\in\cdot)
\end{equation}
for all $n\in\N$. Introducing random variables $T_{0},\wh{T}_{0},T_{1},\wh{T}_{1},...$ and $I_{1},I_{2},...$ on a common probability space $(\Omega,\fA,\Prob)$ such that
\begin{itemize}\itemsep2pt
\item $\Prob(T_{n}\in\cdot)=\Prob(\wh{T}_{n}\in\cdot)=\Prob_{n}(T\in\cdot)$ for each $n\in\N_{0}$,
\item $\Prob(I_{n}\in\cdot)=\Prob_{n}(n-M_{1}\in\cdot)=(p_{n,n-k})_{1\le k\le n}$ for each $n\in\N$,
\item $(T_{n})_{n\ge 0}$, $(\wh{T}_{n})_{n\ge 0}$ and $(I_{n})_{n\ge 1}$ are independent,
\end{itemize}
we may restate \eqref{eq:embarking recursion} in the form of a random recursive equation, namely
\begin{equation}\label{main_rec}
T_{n}\ \eqdist\ 1+\wh{T}_{n-I_{n}}\quad\text{for all }n\ge 1,
\end{equation}
where $T_{0}:=0$ and $\eqdist$ means equality in distribution.

\subsection{Bibliographic notes}\label{subsec:bibliographic}

The random variable $T$ arises in different areas of applied probability and also in varying disguises, for instance, the number of blocks in regenerative compositions \cite{GneIksMar:10,GnedinPitman:05,Gnedin:04}, the number of collisions in exchangeable coalescents \cite{GneIksMar:14,HaasMiermont:11,Pitman:99,Sagitov:99}, the number of positive jumps in a random walk with a barrier \cite{IksMohl:08,MarVerovkin:14}, or the number of cuts needed to isolate the root of a random recursive tree \cite{IksMohl:07}.

Besides the asymptotic analysis of $Q_{n}:=\Prob_{n}(T\in\cdot)$ for particular Markov chains like those appearing in the afore-mentioned (as well as many other) models, there is also work on the behavior of $Q_{n}$ in a general context. Under the assumption that the large jumps of the chain are rare and occur at rates that behave like a negative power of the current state, the scaling limit of $(M_{n})_{n\ge 0}$ was derived in \cite{HaasMiermont:11}. As a byproduct of the main result, the authors also obtained the convergence in distribution of $T_{n}$ properly normalized to the law of $\int_{0}^{\infty} e^{-U_{t}}\,dt$, where $(U_{t})_{t\ge 0}$ is an increasing L\'evy process (subordinator). These results where extended in the recent preprint \cite{BertoinKortch:14} to the case of arbitrary Markov chains (not necessarily decreasing) with negative drift. On the other hand, there are situations where the distributions of the jumps are almost identical for all states far enough from $0$. In such cases it is natural to expect that, given $M_{0}=n$ for large $n$, the trajectory of $(M_{k})_{k\ge 0}$ stays close to the trajectory of $(n-S_{k})_{k\ge 0}$ for a suitable zero-delayed random walk $(S_{k})_{k\ge 0}$ with positive increments, implying that $Q_{n}$ is close to the law of the first passage time $\inf\{k\ge 0:n-S_{k}<a\}$ and therefore to some stable law after normalization. Using such a ``renewal approximation'', the simplest case of a random walk with increments having finite variance, was treated in \cite{VanCutsemYcart:94}, where some sufficient conditions were derived for the convergence of $T_{n}$, properly centered and normalized, to the standard normal law. Finally, we mention a short note \cite{Ross:82}, where a representation of $T_{n}$ as a sum of independent indicators was provided under the assumption that the increment distribution can be decomposed as the product of a function of the current state and a function of the jump size. 


\subsection{Renewal approximation}\label{subsec:renewal approx}

The main purpose of this paper is to provide some general results concerning the distributional convergence of $T_{n}$ in the context of the ``renewal approximation'' mentioned above and to extend this approach to another class of decreasing Markov chains with so-called ``multiplicative'' stationary decrements. More precisely, we will assume that one of the following conditions holds true as $n\to\infty$:
\begin{description}[(Mult)]\itemsep3pt
\item[\AsA] $I_{n}\idist\xi$ for a random variable $\xi$.
\item[\AsM] $n^{-1}I_{n}\idist 1-\eta$ for a $[0,1]$-valued random variable $\eta$.
\end{description}
Here the labels are chosen to serve as mnemonic acronyms for ``additive'' and ``multiplicative''.
Let us stipulate that $\xi$ and $\eta$ are always assumed to be independent of any other occurring random variables.

\vspace{.1cm}
We continue with an outline of the main idea behind our approach. Assume first that condition $\AsA$ holds. Let $(\xi_{n})_{n\ge 1}$ be a sequence of independent copies of $\xi$ and $(S_{n})_{n\ge 0}$ the associated zero-delayed random walk, viz.
\begin{equation}\label{add_rw_def}
S_{0}:=0\quad\text{and}\quad S_{n}:=S_{n-1}+\xi_{n}\quad\text{for }n\ge 1.
\end{equation}
Consider the renewal counting process
\begin{equation}\label{add_rp_def}
N_{n}\ :=\ \sum_{k\ge 0}\1_{\{S_{k}<n\}}\ =\ \inf\{k\ge 0:S_{k}\ge n\},\quad n\in\N_{0},
\end{equation}
thus $N_{0}:=0$. Standard renewal arguments lead to the distributional identity 
\begin{equation}\label{st_renewal_rec}
N_{n}\ \eqdist\ 1+\wh{N}_{n-\xi\wedge n}\quad\text{for }n\ge 1,
\end{equation}
where $(\wh{N}_{n})_{n\ge 0}$ denotes a copy of $(N_{n})_{n\ge 0}$ independent of $\xi$.
Comparing \eqref{main_rec} with \eqref{st_renewal_rec} under the hypothesis $\AsA$, one may expect that the distribution of $T_{n}$ can be approximated for large $n$ by the distribution of $N_{n}$ at least if some additional assumptions on the ``closeness in distribution'' of $I_{n}$ and $\xi\wedge n$ are imposed. In what follows we call the case, when assumption $\AsA$ is in force, the ''additive case``, highlighting that the approximating process $(N_{n})_{n\ge 0}$ in constructed from the standard additive random walk.

\vspace{.1cm}
A similar construction can be given when $\AsM$ holds. Let $(\eta_{n})_{n\ge 1}$ be a sequence of independent copies of $\eta$, $(\Pi_{n})_{n\ge 0}$ the associated \emph{multiplicative random walk}, viz.
\begin{equation}\label{mult_rw_def}
\Pi_{0}:=1\quad\text{and}\quad\Pi_{n}:=\Pi_{n-1}\eta_{n}\quad\text{for }n\ge 1.
\end{equation}
and $(\Lambda_{t})_{t\ge 0}$ the renewal counting process corresponding to $(-\log\Pi_{n})_{n\ge 0}$, thus
\begin{equation}\label{mult_rp_def}
\Lambda_{t}\ :=\ \sum_{k\ge 0}\1_{\{-\log\Pi_{k}\le t\}}\ =\ \inf\{k\ge 0:-\log\Pi_{k}>t\},\quad t\in\R.
\end{equation}
Upon setting $L_{t}:=\Lambda_{\,\log t}$ for $t>0$ and with $(\wh{L}_{t})_{t\ge 0}$ having the obvious meaning, we obtain the distributional identity
\begin{equation}\label{mult_renewal_rec}
L_{t}\ \eqdist\ 1+\wh{L}_{t\eta}\ =\ 1+\wh{L}_{t-t(1-\eta)}\quad\text{for }t\ge 1
\end{equation}
with $L_{t}:=0$ for $0<t<1$, which again looks similar to \eqref{main_rec} for $t=n$, because $I_{n}\approx n(1-\eta)$ for large $n$ if $\AsM$ holds. Since the approximating process $(L_{t})_{t>0}$ emerges here from a
multiplicative random walk we call this situation the ''multiplicative case``.

\subsection{Motivating examples}\label{subsec:examples}

We proceed with a series of examples from applied probability, where ``additive'' or ``multiplicative'' renewal approximations come into play.

\subsubsection*{Coalescents with multiple collisions} 

A coalescent with multiple collisions (or $\Lambda$-coalescent) is a continuous-time Markov  process $(\Sigma^{(n)}_{t})_{t\ge 0}$ on the space of partitions of $\{1,2,\ldots,n\}$ with one type of transitions: if at some time $t\ge 0$ there are $m$ blocks in $\Sigma^{(n)}_{t}$, then each $k$-tuple of them merges into one block at rate
\begin{equation}\label{lambda-rates}
\lambda_{m,k}\ =\ \int_{0}^1 x^{k-2} (1-x)^{m-k}\ \Lambda(dx),\quad 
2\le k\le m,
\end{equation}
where $\Lambda$ is a finite measure on $[0,1]$. Let $(\wh\Sigma^{(n)}_k)_{k\ge 0}$ denote the embedded discrete-time Markov chain at jump epochs and $|\wh\Sigma^{(n)}_k|$ the number of blocks in $\wh\Sigma^{(n)}_k$. Since the transition rates depend on a given state only through its size (number of blocks), it is clear that $(|\wh\Sigma^{(n)}_k|)_{k\ge 0}$ forms a Markov process as well. Moreover, it is decreasing with absorbing state $a=1$, and the total number of collisions, say $C_{n}$, equals the number of transitions of the chain $(|\wh\Sigma^{(n)}_k|)_{k\ge 0}$ before absorption, hence
$$ C_{1}\ =\ 0\quad\text{and}\quad C_{n}\ \eqdist\ 1+ \wh{C}_{n-I_{n}+1}\quad\text{for }n\ge 2, $$
where
\begin{itemize}\itemsep2pt
\item $\wh{C}_{n}$ is a copy of $C_{n}$ for $n\ge 1$; 
\item $I_{n}\in\{2,...,n\}$ denotes a copy of $n-|\wh\Sigma^{(n)}_{1}|+1$ and thus describes the number of blocks involved in the first collision event for $n\ge 2$;
\item the $\wh{C}_{n}$ and $I_{n}$ are independent for $n\ge 2$.
\end{itemize}
For the case when $\Lambda$ is a $\beta(a,b)$-distribution with parameter $(a,b)\in (0,1]\times (0,\infty)$, it is known (see, for instance, formulae (9) and (10) in \cite{GneIksMarMoh:14} or formula (11) in \cite{GneYakub:07}) that
$$ I_{n}-1\ \idist\ I_{\infty}\quad\text{as }n\to\infty, $$
where
$$ \Prob(I_{\infty}=k)\ =\ \frac{(2-a)\Gamma(k+a-1)}{\Gamma(a)(k+1)!}\quad\text{for }k\in\N, $$
and $C_{n}$ exhibits the same weak asymptotics as the renewal process $N_{n}:=\inf\{k\ge 0: S_{k}^{\infty}\ge n\}$, where $(S^{\infty}_{k})_{k\ge 0}$ denotes a random walk with generic increment $I_{\infty}$.\footnote{The stated weak convergence for $I_{n}-1$ holds also for $a\in(1,2)$ but the renewal approximation is no more valid in this range.}

\vspace{.1cm}
We refer to the seminal papers \cite{Pitman:99,Sagitov:99} for the construction and basic properties of $\Lambda$-coalescents, and to the lecture notes \cite{NBerestycki:09} for a survey and further references. More information on the limiting behaviour of $C_{n}$ as well as other functionals of the $\Lambda$-coalescent can be found in \cite{GneIksMar:14}.

\subsubsection*{The Bernoulli sieve}
In a classical occupancy scheme balls are allocated independently in an infinite array of boxes with probability $p_{k}$ of hitting box $k = 1,2,\ldots$ for a fixed probability sequence $(p_{k})_{k\ge 1}$ with positive entries, also called frequencies. The Bernoulli sieve (see \cite{GneIksMar:10,Gnedin:04,GneIksNegRos:09}) forms an extension of the occupancy scheme with random frequencies
$$p_{k}\ :=\ W_{1}W_{2}\cdots W_{k-1}(1-W_{k}),\quad  k\in\N, $$
where the $W_{k}$ are independent copies of a random variable $W\in (0,1)$. If $K_{n}$ denotes the number of occupied boxes after $n$ placements, then (see equation (6) in \cite{Gnedin:04})
$$ K_{0}\ =\ 0\quad\text{and}\quad K_{n}\ \eqdist\ 1+ \wh{K}_{n-J_{n}}\quad\text{for }n\ge 1 $$
where $K_{n},\wh{K}_{n},J_{n}$ are independent, $\wh{K}_{n}\eqdist K_{n}$ and the law of $J_{n}$ equals the conditional law of the number of balls in the first box given that this number is positive, viz.
$$\Prob(J_{n}=k)\ =\ \frac{1}{1-\Erw W^{n}}\binom{n}{k}\Erw (1-W)^{k}W^{n-k}\quad\text{for } k=1,2,\ldots,n. $$ 
By the law of large numbers,
$$ n^{-1}J_{n}\ \idist\ 1-W\quad\text{as }n\to\infty, $$
and, under additional assumptions on the rate of this convergence, $K_{n}$ has the same weak asymptotics as
$L_{n}:=\inf\{k\ge 0: W_{1} W_{2}\cdots W_{k}<1/n\}$. This has been proved in \cite{GneIksNegRos:09} if $\Erw|\log (1-W)|<\infty$, while the case of infinite mean was treated in \cite{GneIksMar:10} by using different approach.

\subsubsection*{Random walks with a barrier}

Let $\zeta_{1},\zeta_{2},\ldots$ be independent copies of a random variable $\zeta$ taking values in $\N$. The associated random walk $(R_{k}^{(n)})_{k\ge 0}$ with barrier $n\in\N$ is then defined as follows:
$$R_{0}^{(n)}:=\ 0\quad\text{and}\quad R_{k}^{(n)}:=\ R_{k-1}^{(n)}+\zeta_{k}\1_{\{R_{k-1}^{(n)}+\zeta_{k}<n\}}\quad\text{for }k\ge 1. $$ 
Obviously, $(n-R_{k}^{(n)})_{k\ge 0}$ forms a nonincreasing Markov chain on $\N$ with absorbing state 1 which is eventually attained under the additional condition that $\Prob(\zeta=1)>0$. Defining the number of jumps of $(R_{k}^{(n)})_{k\ge 0}$, viz.
$$ Z_{n}\ :=\ \sum_{k\ge 1}\1_{\{R_{k-1}^{(n)}\ne R_{k}^{(n)}\}}\ =\ \sum_{k\ge 1}\1_{\{R_{k-1}^{(n)}+ \zeta_{k}<n\}}, $$
one can easily see that
$$ Z_{1}\ =\ 0,\quad\text{and}\quad Z_{n}\ \eqdist\ 1+\wh{Z}_{n-\vth_{n}}\quad\text{for }n\ge 2,
$$
where as usual $Z_{n},\wh{Z}_{n}$ and $\vth_{n}$ are independent, $\wh{Z_{n}}\eqdist Z_{n}$ and the law of $\vth_{n}$ equals the conditional law of $\zeta$ given $\zeta<n$, thus
$$ \Prob(\vth_{n}=k)\ =\ \Prob\{\zeta=k|\zeta<n)\quad\text{for }k=1,2,\ldots,n-1. $$
Using the obvious fact that
$$ \vth_{n}\ \idist\ \zeta\quad\text{as }n\to\infty $$
the corresponding renewal approximation has been obtained in \cite{IksMohl:08} under the assumption that $\zeta$ lies in the domain of attraction of a stable law of index $\alpha\in[1,2]$.

\vspace{.1cm}
The number of zero jumps of $(R_{k}^{(n)})_{k\ge 0}$ (those dismissed in the unrestricted walk because they would lead to value greater or equal to $n$), i.e.
$$ V_{n}\ := \sum_{k=1}^{T_{n}}\1_{\{R_{k-1}^{(n)}+ \zeta_{k}\ge n\}} $$
with $T_{n}:=\inf\{k\in\N_{0}:R_{k}^{(n)}=n-1\}$ being the absorption time, provides another functional of interest and a further example involving a multiplicative renewal approximation. In \cite{MarVerovkin:14}, it was shown that the sequence $V_{n}':=V_{n}+\1_{\{n>1\}}$ satisfies
$$ V_{1}'\ =\ 0,\quad\text{and}\quad V_{n}'\ \eqdist\ 1+\wh{V}_{Y_{n}}'\quad\text{for }n\ge 2, $$
where $Y_{n}$ denotes the undershoot at $n$ of a standard random walk with generic increment $\zeta$ and the usual independence assumptions are made. Further assuming that
$\Prob\{\zeta>n\}=cn^{-\alpha}+O(n^{-(\alpha+\eps)})$ for some $c>0$, $\alpha\in(0,1)$ and $\eps>0$ and using the classical observation due to Dynkin that
$$ n^{-1}Y_{n}\ \idist\ \eta_{\alpha}\quad\text{as }n\to\infty, $$
where $\eta_{\alpha}$ has a $\beta(1-\alpha,\alpha)$-distribution, it was proved in \cite{MarVerovkin:14} that a renewal approximation can be used with a multiplicative random walk having generic step size $\eta_{\alpha}$.

\subsubsection*{A simple decreasing Markov chain} 

Our last example shows that the convergence condition $\AsA$ alone does not suffice for the renewal approximation to work. In fact, the distributions of $X_{n}$ and $N_{n}$ may exhibit a completely different asymptotic behaviour as the following example from \cite{Marynych:10} demonstrates. Consider a decreasing Markov chain $(M_{k})_{k\ge 0}$ with absorbing state $0$, transition probabilities
$$ p_{i,0}\ =\ 1-p_{i,i-1}\ =\ \frac{1}{i},\quad i\in\N, $$
absorption time $T$ and random variables $T_{n},\wh{T}_{n},I_{n}$ as defined at the beginning of this section, thus satisfying \eqref{main_rec}. Then one can easily verify that 
$$ \Prob(I_{n}=n)\ =\ 1-\Prob(I_{n}=1)\ =\ \frac{1}{n} $$
for all $n\ge 1$ and particularly
$$ I_{n}\ \iprob\ \xi\ =\ 1\quad\text{as }\quad n\to\infty $$
Consequently, the corresponding renewal process with generic step size $\xi$ is degenerate.
On the other hand, it can be checked using generating functions that $T_{n}$ has a uniform law on $\{1,2,\ldots,n\}$ for all $n\in\N$ whence
$$ \frac{T_{n}}{n}\ \idist\ \textit{Unif}\,(0,1),\quad n\to\infty. $$

\subsection{Minimal $L^{p}$-distance}

As illustrated by the last example, assumption $\AsA$ alone does not suffice for our renewal approximation to work, and the same is true for $\AsM$. Indeed, some extra conditions on the rate of convergence of $I_{n}$ to $\xi$ in the additive case, and of $n^{-1}I_{n}$ to $1-\eta$ in the multiplicative case are necessary. 

For the approach of this paper, which has already been used in \cite{GneIksMarMoh:14} and \cite{MarVerovkin:14}, the rate of convergence in $\AsA$ and $\AsM$ will be measured in terms of the minimal $L^{p}$-distance, for the following reasons a natural choice: 
\begin{itemize}
\item the convergence in the chosen distance implies convergence in distribution.
\item the distance is invariant, in a certain sense, under affine transformations of the laws, thus making calculations with centered and/or normalized random variables easy.
\end{itemize}

Let us briefly recall the definition and basic properties of the minimal $L^{p}$-distance, defined on the set $\mathcal{D}^{p}$ of probability distributions on $\R$ with finite absolute $p$-th moment. 
A pair of random variables $(X,Y)$, defined on a common probability space, is called a $(F,G)$-coupling for $F,G\in\cD^{p}$, if $\cL(X)=F$ and $\cL(Y)=G$, where $\cL(X)$ denotes the law of $X$. We then write $(X,Y)\sim (F,G)$. The minimal $L^{p}$-distance between $F$ and $G$ is now defined by
\begin{equation}\label{min_l_{p}_def}
d_{p}(F,G)\ :=\ \inf_{(X,Y)\sim (F,G)}\left(\Erw |X-Y|^{p}\right)^{1\wedge 1/p}\ =\ \inf_{(X,Y)\sim (F,G)}\|X-Y\|_{p}.
\end{equation}
In what follows, we also write, in slight abuse of language, $d_{p}(X,Y)$ and $d_{p}(X,G)$ for $d_{p}(\cL(X),\cL(Y))$ and $d_{p}(\cL(X),G)$, respectively.

\vspace{.1cm}
The following properties of $d_{p}$ for $p\ge 1$, summarized for our convenience, are well known (see, for instance \cite{Alsmeyer:13,GivensShortt:84,JohnsonSamworth:05,Rachev:91,Zolotarev:97}).

\begin{proposition}\label{was_dis_{p}rop}
Let $p\ge 1$ and $X,Y$ be random variables with laws $F,G\in\cD^{p}$, respectively. Further, let $F^{\leftarrow}(x):=\inf\{y:F(y)\ge x\}$ denote the pseudo-inverse of $F$ and $U$ a $\textit{Unif}\,(0,1)$ random variable. The function $d_{p}(\cdot,\cdot)$ has the following properties:
\begin{description}[(P4)]\itemsep2pt
\item[(P1)] The infimum in Equation \eqref{min_l_{p}_def} is attained for the $(F,G)$-coupling $(F^{\leftarrow}(U),G^{\leftarrow}(U))$, thus
$$ d_{p}(X,Y)\ =\ \left(\int_{0}^{1}|F^{\leftarrow}(x)-G^{\leftarrow}(x)|^{p}\ dx\right)^{1/p}. $$
In particular,
$$ d_{1}(X,Y)\ =\ \int_{0}^{1}|F^{\leftarrow}(x)-G^{\leftarrow}(x)|\ dx\ =\ \int_{\R}|F(x)-G(x)|\ dx. $$
\item[(P2)] If $p=1$, \textbf{Kantorovich-Rubinstein representation}\footnote{A representation of this form does not exist for $p>1$, see \cite[Lemma 4.3.2]{Rachev:91}, but it holds for $p\in(0,1)$ with $|x-y|$ replaced by $|x-y|^{p}$ in the definition of the set $\mathcal{F}$.} holds, viz.
$$ d_{1}(X,Y)=\sup_{\cF}|\Erw f(X)-\Erw f(Y)|, $$
where $\cF$ denotes the class of all Lipschitz functions $f:\R\to\R$ with Lipschitz constant one, that is $|f(x)-f(y)|\le |x-y|$ for all $x,y\in\R$.
\item[(P3)] $d_{p}(X+Z,Y+Z)\le d_{p}(X,Y)$ for any further random variable $Z\in{\mathcal D}_{p}$ independent of $(X,Y)$.
\item[(P4)] $d_{p}(aX+b,aY+b)=|a|d_{p}(X,Y)$ for all $a,b\in\R$.
\item[(P5)] If $(X_{n})_{n\ge 1}$ denotes a sequence of random variables with laws in $\cD^{p}$, then $d_{p}(X_{n},X)\to 0$ holds iff $X_{n}\idist X$ and $\Erw|X_{n}|^{p}\to\Erw|X|^{p}$, as $n\to\infty$.
\end{description}
\end{proposition}

The rest of the paper is organized as follows. Section \ref{main_res_sec} contains our main results, which are stated as Theorem \ref{main_thm1} and Theorem \ref{main_thm2}. The proofs can be found in Section \ref{proofs_sec} and some necessary auxiliary results, including the ones on the convergence of the renewal processes $(N_{n})_{n\ge 0}$ and $(\Lambda_{t})_{t\ge 0}$ in minimal $L^{p}$-distance, are collected in an appendix.

\section{Main results}\label{main_res_sec}

\subsection{Weak convergence in the additive case}

\begin{theorem}\label{main_thm1}
Suppose that $\AsA$ holds and the law of $\xi$ is $1$-arithmetic and nondegenerate with finite mean $\mu$.
\begin{description}[(c)]\itemsep2pt
\item[(a)] If $\sigma^{2}:=\Var\xi<\infty$ and
\begin{equation}\label{additive_finite_variance_condition}
d_{2}(I_{n},\xi\wedge n)\ =\ o(n^{-1/2})\quad\text{as }n\to\infty,
\end{equation}
then
\begin{equation*}
d_{2}\left(\frac{T_{n}-\mu^{-1}n}{\sigma\mu^{-3/2}n^{1/2}},\,\cN(0,1)\right)\ \stackrel{n\to\infty}{\longrightarrow}\ 0,
\end{equation*}
where $\cN(0,1)$ denotes the standard normal law.
\item[(b)] If $\sigma^{2}=\infty$,
\begin{equation*}
\ell(n)\ :=\ \Erw [\xi^{2}\1_{\{\xi\le n\}}]
\end{equation*}
is slowly varying at infinity and
\begin{equation}\label{additive_sv_variance_condition}
d_{1}(I_{n},\xi\wedge n)\ =\ o(n^{-1}c(n)),\quad\text{as }n\to\infty
\end{equation}
for a positive function $c(t)$ such that
\begin{equation}\label{rv1_normalization_additive_case}
\lim_{n\to\infty}\frac{n\ell(c(n))}{c(n)^{2}}\ =\ 1,
\end{equation}
then
\begin{equation*}
d_{1}\left(\frac{T_{n}-\mu^{-1}n}{\mu^{-3/2}c(n)},\,\cN(0,1)\right)\ \stackrel{n\to\infty}{\longrightarrow}\ 0.
\end{equation*}
\item[(c)] If $\ell(n):=n^{\alpha}\,\Prob(\xi\ge n)$ is slowly varying at infinity for some $\alpha\in (1,2)$ and Condition \eqref{additive_sv_variance_condition} holds for a positive function $c(t)$ satisfying
\begin{equation}\label{rv2_normalization_additive_case}
\lim_{n\to\infty}\frac{n\ell(c(n))}{c(n)^{\alpha}}\ =\ 1,
\end{equation}
then
\begin{equation*}
d_{1}\left(\frac{T_{n}-\mu^{-1}n}{\mu^{-(\alpha+1)/\alpha}c(n)},\,\cS_{\alpha}\right)\ \stackrel{n\to\infty}{\longrightarrow}\ 0,
\end{equation*}
where $\cS_{\alpha}$ denotes the $\alpha$-stable law with characteristic function
\begin{equation}\label{cf_stable_law}
t\ \mapsto\ \exp\big[-|t|^{\alpha}\Gamma(1-\alpha)(\cos(\pi\alpha/2)+{\rm i}\sin(\pi\alpha/2)\sign(t))\big].
\end{equation}
\end{description}
\end{theorem}

\begin{remark}\label{rem1}
Regarding the existence and further properties of the function $c$ appearing in parts (b) and (c) of Theorem \ref{main_thm1} as well as Theorem \ref{main_thm2} we refer to Remark \ref{c_t_existence_and_properties} in the Appendix.
\end{remark}

\subsection{Weak convergence in the multiplicative case}

The case of multiplicative renewal approximation is treated by our second theorem.

\begin{theorem}\label{main_thm2}
Suppose that $\AsM$ holds and the distribution of $|\log\eta|$ is nonarithmetic with finite mean $\mu_{0}$.
\begin{description}[(c)]\itemsep2pt
\item[(a)] If $\sigma_{0}^{2}:=\Var(|\log \eta|)<\infty$ and 
\begin{equation}\label{multiplicative_finite_variance_condition}
d_{1}\left(\log\left(1-\frac{I_{n}}{n}\right),\log\eta\right)\ =\ o\left(\frac{1}{\log^{1/2}n}\right)\quad\text{as }n\to\infty,
\end{equation}
then
\begin{equation*}
d_{1}\left(\frac{T_{n}-\mu_{0}^{-1}\log n}{\sigma_{0}\mu_{0}^{-3/2}\log^{1/2}n},\,\cN(0,1)\right)\ \stackrel{n\to\infty}{\longrightarrow}\ 0.
\end{equation*}
\item[(b)] If $\sigma_{0}^{2}=\infty$,
\begin{equation}
\ell(t)\ :=\ \Erw[\log^{2}\eta\1_{\{|\log \eta|\le t\}}]
\end{equation}
is slowly varying at infinity and 
\begin{equation}\label{mult_sv_variance_condition}
d_{1}\left(\log\Big(1-\frac{I_{n}}{n}\Big),\log \eta\right)\ =\ o\left(\frac{c(\log n)}{\log n}\right)\quad\text{as }n\to\infty
\end{equation}
for a positive function $c(t)$ satisfying
\begin{equation}\label{rv1_{n}ormalization_mult_case}
\lim_{t\to\infty}\frac{t\ell(c(t))}{c(t)^{2}}\ =\ 1,
\end{equation}
then
\begin{equation*}
d_{1}\left(\frac{T_{n}-\mu_{0}^{-1}\log n}{\mu_{0}^{-3/2}c(\log n)},\,\cN(0,1)\right)\ \stackrel{n\to\infty}{\longrightarrow}\ 0.
\end{equation*}
\item[(c)] If $\ell(t):=t^{\alpha}\,\Prob(|\log \eta| > t)$ is slowly varying at infinity for some $\alpha\in(1,2)$ and Condition \eqref{mult_sv_variance_condition} holds for some positive function $c(t)$ satisfying
\begin{equation}\label{rv2_{n}ormalization_mult_case}
\lim_{t\to\infty}\frac{t\ell(c(t))}{c(t)^{\alpha}}\ =\ 1,
\end{equation}
then
\begin{equation*}
d_{1}\left(\frac{T_{n}-\mu_{0}^{-1}\log n}{\mu_{0}^{-(\alpha+1)/\alpha}c(\log n)},\,\cS_{\alpha}\right)\ \stackrel{n\to\infty}{\longrightarrow}\ 0,
\end{equation*}
where $\mathcal{S}_{\alpha}$ is the $\alpha$-stable law with characteristic function in \eqref{cf_stable_law}.
\end{description}
\end{theorem}

\section{Proofs}\label{proofs_sec}

\subsection{Proof of Theorem \ref{main_thm1}}\label{sec_{p}roof_{t}hm1}

Recall that $(N_{n})_{n\ge 0}$ denotes the renewal counting process of a random walk $(S_{n})_{n\ge 0}$ with generic step size $\xi$. We start by pointing out that
\begin{equation}\label{mom_{i}ndex_conv}
\lim_{n\to\infty}\Erw I_{n}\ =\ \Erw \xi\ =\ \mu.
\end{equation}
Indeed with $p=2$ in (a) and $p=1$ in (b),(c), we have
$$ d_{p}(I_{n},\xi)\ \le\ d_{p}(I_{n},\xi\wedge n)\ +\ d_{p}(\xi\wedge n,\xi), $$
where the first summand tends to zero by \eqref{additive_finite_variance_condition} and \eqref{additive_sv_variance_condition}, respectively, while the second one does so by the dominated convergence theorem. Hence \eqref{mom_{i}ndex_conv} follows by (P5) in Proposition \ref{was_dis_{p}rop}, and it ensures that $\psi_{n}=n$ satisfies Condition (C1) of Lemma \ref{boundedness} in the Appendix with $p_{n,k}:=\Prob(I_{n}=n-k)$, a fact to be used further below (see before \eqref{asymptotics eps_n}).

In what follows, we let $c(t)$ be given by $\sigma t^{1/2}$ in part (a), by \eqref{rv1_normalization_additive_case} in part (b), and by \eqref{rv2_normalization_additive_case} in part (c). We also put $\alpha=2$ in (a) and (b), and we write $d_{1(2)}$ to denote $d_{p}$ with $p=2$ in (a) and $p=1$ in (b) and (c).
Finally, let $\cW$ be the standard normal law $\cN(0,1)$ in (a) and (b) and the stable law $\cS_{\alpha}$ with characteristic function \eqref{cf_stable_law} in (c). Our task is then to show that
$$ d_{1(2)}\left(\frac{T_{n}-\mu^{-1} n}{\mu^{-(\alpha+1)/\alpha} c(n)}, \cW\right) $$
converges to 0 as $n\to\infty$. Using the triangle inequality, we can bound this distance by
\begin{align*}
d_{1(2)}\left(\frac{T_{n}-\mu^{-1} n}{\mu^{-(\alpha+1)/\alpha} c(n)}, \frac{N_{n}-\mu^{-1} n}{\mu^{-(\alpha+1)/\alpha} c(n)}\right)\ +\ d_{1(2)}\left(\frac{N_{n}-\mu^{-1} n}{\mu^{-(\alpha+1)/\alpha} c(n)}, \cW\right),
\end{align*}
and Proposition \ref{renewal_{p}rocess_d_{p}_conv} (1-arithmetic case) ensures that the second term converges to zero. As for the first one, it is enough to prove that 
\begin{equation}\label{proof_add_case_est1}
e_{n}\ :=\ d_{1(2)}(T_{n},N_{n})\ =\ o(c(n))\quad\text{as }n\to\infty
\end{equation}
by property (P4) in Proposition \ref{was_dis_{p}rop}.

\vspace{.1cm}
Using the recursions \eqref{main_rec}, \eqref{st_renewal_rec} and again property (P4), we have for $n\ge 1$ 
\begin{align*}
e_{n}\ &=\ d_{1(2)}(\wh{T}_{n-I_{n}},\wh{N}_{n-\xi\wedge n})\\
&\le\ d_{1(2)}(\wh{N}_{n-I_{n}},\wh{N}_{n-\xi\wedge n})\ +\ d_{1(2)}(\wh{T}_{n-I_{n}},\wh{N}_{n-I_{n}})\\
&\le\ d_{1(2)}(\wh{N}_{n-I_{n}},\wh{N}_{n-\xi\wedge n})\ +\ \sum_{k=0}^{n-1}\Prob(I_{n}=n-k)\,d_{1(2)}(T_{k},N_{k})\\
&=\ e_{n}'\ +\ \sum_{k=0}^{n-1}\Prob(I_{n}=n-k)\,e_{k}\quad\text{with}\quad e_{n}':=d_{1(2)}(\wh{N}_{n-I_{n}},\wh{N}_{n-\xi\wedge n}).
\end{align*}

Assuming we have already proved
\begin{equation}\label{proof_add_case_est2}
e_{n}'=o(n^{-1}c(n))\quad\text{as }n\to\infty,
\end{equation}
we can apply Lemma \ref{L61gen} with $\psi_{n}=n$ and $r_{n}= \eps c(n)/n$ for arbitrarily small $\eps>0$ to infer that, as $n\to\infty$,
\begin{equation}\label{asymptotics eps_n}
e_{n}\ =\ \eps\,O\left(\sum_{k=1}^{n}\sup_{j\ge k}j^{-1}c(j)\right)\ =\ \eps\,O\left(\sum_{k=1}^{n}k^{-1}c(k)\right)\ =\ \eps\,O(c(n)),
\end{equation}
where we have utilized Theorem 1.5.3 and Proposition 1.5.8 in \cite{BingGolTeug:89} and the fact that $c(x)$ is a regularly varying function of index $1/\alpha$, see Remark \ref{c_t_existence_and_properties} in Appendix.

\vspace{.1cm}
Left with the proof of \eqref{proof_add_case_est2}, let $(I_{n}',\xi')$ be a $(\mathcal{L}(I_n),\mathcal{L}(\xi))$-coupling such that
$d_{1(2)}(I_{n},\xi\wedge n)=d_{1(2)}(I_{n}',\xi'\wedge n)=\|I_{n}'-\xi'\wedge n\|_{1(2)}$ and 
$(S_{n}')_{n\ge 0}$ be a copy of $(S_{n})_{n\ge 0}$ which is independent of $(I_{n}',\xi')$. Denote by $(N_{n}')_{n\ge 0}$ the corresponding renewal counting process, clearly a copy of $(N_{n})_{n\ge 0}$. Then
\begin{align*}
e_{n}'\ &=\ d_{1(2)}(\wh{N}_{n-I_{n}},\wh{N}_{n-\xi\wedge n})\\
&\le\ \|N_{n-I_{n}'}'-N_{n-\xi'\wedge n}'\|_{1(2)}\\
&=\ \left\|\sum_{k\ge 0}\1_{\{(n-I_{n}')\wedge (n-\xi'\wedge n)< S_{k}'\le (n-I_{n}')\vee (n-\xi'\wedge n)\}}\right\|_{1(2)}\\
&\le\ \|I_{n}'-\xi'\wedge n\|_{1(2)},
\end{align*}
where the last inequality follows from the fact that, in view of $\Prob(\xi_{1}\ge 1)=1$, the number of points $S_{k}'$ falling in some interval $(a,b]$, $a,b\in\N$, cannot exceed the length of this interval. Consequently,
$$ e_{n}'\ \le\ d_{1(2)} (I_{n}',\xi'\wedge n)\ =\ d_{1(2)}(I_{n},\xi\wedge n), $$
and thus \eqref{proof_add_case_est2} by \eqref{additive_finite_variance_condition} in part (a) and by \eqref{additive_sv_variance_condition} in parts (b) and (c). This completes the proof.

\subsection{Proof of Theorem \ref{main_thm2}}\label{sec_{p}roof_{t}hm2}

The proof of Theorem \ref{main_thm2} uses similar ideas as the previous one in the additive case, but for technical reasons it is more convenient to work with the stationary version of the renewal counting process $(\Lambda_{t})_{t\ge 0}$ associated with $(-\log\Pi_{n})_{n\ge 0}$ (see \eqref{mult_rw_def} and \eqref{mult_rp_def} for the definition of $\Pi_{n}$ and $\Lambda_{t}$).

Let $\eta_{0}^{*}\in(0,1)$ be a random variable independent of $(\Pi_{k})_{k\ge 0}$  such that
\begin{equation}\label{overshoot_density}
r(t)\ :=\ \Prob(|\log \eta_{0}^{*}|\le t)\ =\ \frac{1}{\mu_{0}}\int_{0}^{t}\Prob(|\log \eta|>s)\ ds,\quad t\ge 0.
\end{equation}
Define the delayed multiplicative random walk $(\Pi_{k}^{*})_{k\ge 0}$ by
$$ \Pi_{0}^{*}\ =\ \eta_{0}^{*}\quad\text{and}\quad \Pi_{k}^{*}:=\eta_{0}^{*}\eta_{1}\cdots\eta_{k}\quad\text{for }k\in\N, $$
the stationary renewal counting process associated with $(-\log\Pi^{*}_{k})_{k\ge 0}$ by
$$ \Lambda_{t}^{*}\ :=\ \sum_{k\ge 0}\1_{\{-\log \Pi_{k}^{*}\le t\}},\quad t\in\R, $$
and finally $L_{t}^{*}:=\Lambda_{\log t}^{*}$ for $t>0$. Then it is well-known that
\begin{equation}\label{st_ren_func}
\Erw\Lambda_{t}^{*}\ =\ \frac{t^{+}}{\mu_{0}}\quad\text{for all }t\in\R,
\end{equation}
a fact frequently be used hereafter. Moreover, as $\Lambda_{t}=\inf\{k:-\log\Pi_{k}>t\}$, Wald's identity ensures (see e.g. \cite[Theorem 2.5.1]{Gut:09})
\begin{equation}\label{gen_ren_func}
\Erw\Lambda_{t}\ =\ \frac{t}{\mu_{0}}\ +\ o(t)\quad\text{as }t\to\infty,
\end{equation}
and $o(t)$ may be replaced with $o(c(t))$ under the assumptions of part (b) and (c) of Theorem \ref{main_thm2} (see p.~5 in \cite{IksMarMei:12}). The counterpart of \eqref{mult_renewal_rec} for the process $(L_{t}^{*})_{t\ge 0}$ is given by
\begin{equation}\label{mult_renewal_st_rec}
L_{t}^{*}\ \eqdist\ \1_{\{\eta_{0}^{*}>1/t\}}+\wh{L}_{t\eta}^{*},\quad t\ge 0.
\end{equation}
where $(\wh{L}_{t}^{*})_{t>0}$ is independent of $\eta$ and obtained as a copy of $(L_{t}^{*})_{t\ge 0}$ by replacing the $\eta_{1},\eta_{2},\ldots$ by i.i.d. copies in the definition of the underlying multiplicative random walk $(\Pi_{k}^{*})_{k\ge 0}$ while keeping the delay $\eta_{0}^{*}$ fixed.

\vspace{.2cm}
Returning to the proof of Theorem \ref{main_thm2}, we consider again all three parts (a)-(c) simultaneously and use analogous notation as in Section \ref{sec_{p}roof_{t}hm1}. This means that $c(t)$ equals $\sigma_{0}t^{1/2}$ in part (a), is given by \eqref{rv1_{n}ormalization_mult_case} in part (b) and given by \eqref{rv2_{n}ormalization_mult_case} in part (c), $\alpha=2$ in parts (a) and (b). Also, $\mathcal{W}$ stands for the limiting distribution, thus for $\cN(0,1)$ in (a) and (b), and for $\cS_{\alpha}$ in (c).

Using the triangle inequality, we obtain
\begin{align*}
d_{1}\left(\frac{T_{n}-\mu_{0}^{-1} \log n}{\mu_0^{-(\alpha+1)/\alpha} c(\log n)}, \mathcal{W}\right)\ &\le\ d_{1}\left(\frac{T_{n}-\mu_{0}^{-1} \log n}{\mu_0^{-(\alpha+1)/\alpha} c(\log n)}, \frac{L^{*}_{n}-\mu_{0}^{-1} \log n}{\mu_0^{-(\alpha+1)/\alpha}c(\log n)}\right)\\
&+\ d_{1}\left(\frac{L^{*}_{n}-\mu_{0}^{-1} \log n}{\mu_0^{-(\alpha+1)/\alpha} c(\log n)}, \frac{L_{n}-\mu_{0}^{-1} \log n}{\mu_0^{-(\alpha+1)/\alpha} c(\log n)}\right)\\
&\quad +d_{1}\left(\frac{L_{n}-\mu_{0}^{-1} \log n}{\mu_0^{-(\alpha+1)/\alpha} c(\log n)}, \mathcal{W}\right).
\end{align*}
The third summand converges to zero by Proposition \ref{renewal_{p}rocess_d_{p}_conv} so that it is enough to prove (use (P4) of Proposition \ref{was_dis_{p}rop})
\begin{equation}\label{proof_mult_case_est1}
d_{1}(T_{n},L^{*}_{n})\ =\ o(c(\log n))\quad\text{as }n\to\infty
\end{equation}
and
\begin{equation}\label{proof_mult_case_est2}
d_{1}(L_{n},L^{*}_{n})\ =\ o(c(\log n))\quad\text{as }n\to\infty.
\end{equation}

We first consider \eqref{proof_mult_case_est2} and recall that the renewal counting process $(\Lambda_{t})_{t\in\R}$ is subadditive in distribution, viz.
\begin{equation}\label{subadditive}
\Lambda_{u+v}-\Lambda_{u}\ \overset{d}{\leqslant}\ \Lambda_{v}\quad\text{for all }u,v\in\R,
\end{equation}
where $\overset{d}{\leqslant}$ denotes stochastic ordering: $X\overset{d}{\leqslant} Y$ iff $\Prob(X>x)\le\Prob(Y>x)$ for all $x\in\R$.

In order to prove \eqref{proof_mult_case_est2}, notice that
\begin{equation}\label{stationary_from_standard}
\Lambda_{t}^{*}\ =\ \Lambda_{t-|\log\eta_{0}^{*}|}\quad\text{for all }t\in\R
\end{equation}
with $\eta_{0}^{*}$ being independent of $(\Lambda_{t})_{t\ge 0}$. Using this, we infer
\begin{align*}
d_{1}(L_{n},L^{*}_{n})\ &=\ d_{1}(\Lambda_{\log n},\Lambda_{\,\log n}^{*})\ \le\ \Erw|\Lambda_{\log n}-\Lambda_{\,\log n}^{*}|\\
&=\ \Erw \Lambda_{\,\log n}-\Erw\Lambda_{\,\log n}^{*}\ =\ o(c(\log n)),
\end{align*}
where $\Lambda_{t}^{*}\le\Lambda_{t}$ for all $t\in\R$ by \eqref{stationary_from_standard} has been utilized for the penultimate equality, and \eqref{st_ren_func}, \eqref{gen_ren_func} plus subsequent remark for the final estimate.  

\vspace{.1cm}
Left with the proof of \eqref{proof_mult_case_est1}, use the recursions \eqref{mult_renewal_st_rec} and \eqref{main_rec} to find that
\begin{align*}
e_{n}\ &:=\ d_{1}(T_{n},L^{*}_{n})\\
&=\ d_{1}(1+\wh{T}_{n-I_{n}},\1_{\{\eta_{0}^{*}>1/n\}}+\wh{L}^{*}_{n\eta})\\
&=\ d_{1}(\wh{T}_{n-I_{n}},\wh{L}^{*}_{n\eta}-\1_{\{\eta_{0}^{*}\le 1/n\}})\\
&\le\ d_{1}(\wh{L}^{*}_{n-I_{n}},\wh{L}^{*}_{n\eta}-\1_{\{\eta_{0}^{*}\le 1/n\}})\ +\ d_{1}(\wh{T}_{n-I_{n}},\wh{L}^{*}_{n-I_{n}})\\
&=\ e_{n}'\ +\ \sum_{k=0}^{n-1}\Prob(I_{n}=n-k)\,e_{k},
\end{align*}
where $e_{n}':=d_{1}(\wh{L}^{*}_{n-I_{n}},\wh{L}^{*}_{n\eta}-\1_{\{\eta_{0}^{*}\le 1/n\}})$. Let us assume for a moment that 
\begin{equation}\label{proof_mult_case_est3}
e_{n}'\ =\ o\left(\frac{c(\log n)}{\log n}\right)\quad\text{as }n\to\infty
\end{equation}
is already known and further note that $\AsM$ yields
\begin{equation}\label{mom_index_conv_mult}
\Erw (n-I_{n})\ \simeq\ n\,\Erw |\log \eta|\ =\ \mu_{0} n\quad\text{as }n\to\infty.
\end{equation}
The latter implies that the sequence $\psi_{n}\equiv 1$ satisfies condition (C1) of Lemma \ref{boundedness} when putting $p_{n,k}:=\Prob(I_{n}=n-k)$. An appeal to Lemma \ref{L61gen} together with \eqref{proof_mult_case_est3} provides us with
$$ e_{n}\ =\ O\left(\sum_{k=1}^{n}\sup_{j\ge k}\frac{e_j'}{j}\right)\ =\ o\left(\sum_{k=3}^{n}\sup_{j\ge k}\frac{c(\log j)}{j\log j}\right)\ =\ o\left(\sum_{k=3}^{n}\frac{c(\log k)}{k\log k}\right),
$$
where we have used \cite[Theorem 1.5.3]{BingGolTeug:89} and the fact that $c(\log t)/(t\log t)$ is regularly varying with index $-1$ (see Remark \ref{c_t_existence_and_properties} in the Appendix). Note that the sum $\sum_{k=3}^{n}\frac{c(\log n)}{n\log n}$ diverges as $n$ tends to $\infty$, because $c(t)$ varies regularly with index $1/\alpha>0$. Consequently, as $n\to\infty$,
$$ \sum_{k=3}^{n}\frac{c(\log k)}{k\log k}\ \simeq\ \int_{e}^{n}\frac{c(\log y)}{y\log y}\ dy\ =\ \int_{1}^{\log n}\frac{c(u)}{u}\ du\ \simeq\ \text{const}\cdot c(\log n) $$
by \cite[Theorem 1.6.1]{BingGolTeug:89} and \eqref{proof_mult_case_est1} is proved.

\vspace{.1cm}
It remains to show that \eqref{proof_mult_case_est3} holds. Let $(I_{n}',\eta')$ be a $(\cL(I_{n}),\cL(\eta))$-coupling such that 
$$ d_{1}\left(\log \left(1-\frac{I_{n}}{n}\right),\log\eta\right)\ =\ \left\|\log \left(1-\frac{I_{n}'}{n}\right)-\log\eta'\right\|_{1} $$ 
and $(\log\wh\Pi_{n}^{*})_{n\ge 0}$ be a copy of $(\log\Pi_{n}^{*})_{n\ge 0}$ with delay variable $\wh\eta_{0}^{*}$ and independent of $(I_{n}',\eta')$. Further defining 
$$ \wh\Lambda^{*}_{t}\ :=\ \sum_{k\ge 0}\1_{\{-\log\wh\Pi_{k}^{*}\le t\}}\quad\text{for } t\in\R $$ 
and $\wh{L}^{*}_{t}:=\wh\Lambda^{*}_{\log t}$ for $t>0$, we have
\begin{align*}
e_{n}'\ &=\ d_{1}(\wh{L}^{*}_{n-I_{n}},\wh{L}^{*}_{n\eta}-\1_{\{\eta_{0}^{*}\le 1/n\}})\\
&\le\ \left\|\wh{L}^{*}_{n-I_{n}'}-\wh{L}^{*}_{n\eta'}-\1_{\{\wh\eta_{0}^{*}\le 1/n\}}\right\|_{1}\\
&\le\ \|\wh{L}^{*}_{n-I_{n}'}-\wh{L}^{*}_{n\eta'}\|_{1}\ +\ \Prob(\wh\eta_{0}^{*}\le 1/n),
\end{align*}
where $\wh\eta_{0}^{*}=\wh\Pi_{0}^{*}\eqdist \eta_{0}^{*}$. Recalling \eqref{overshoot_density}, we see that
\begin{align}
\begin{split}\label{r_vs_c}
\Prob(\eta_{0}^{*}\le 1/n)\ &=\ \Prob(-\log \eta_{0}^{*}\ge \log n)\\
&=\ \frac{1}{\mu_{0}}\int_{\log n}^{\infty}\Prob(|\log \eta|>s)\ ds\ =\ 1-r(\log n).
\end{split}
\end{align}
Since in all parts\footnote{In (a) the numerator is bounded, since $1-r(t)$ is integrable, while in (b),\,(c) the numerator varies regularly with index $2-\alpha$ and the denominator varies regularly with index $1/\alpha>2-\alpha$ (with $\alpha=2$ in parts (a) and (b)).} (a)-(c)
$$ \lim_{t\to\infty}\frac{t(1-r(t))}{c(t)}=0, $$
we find that
$$ \Prob(\eta_{0}^{*}\le 1/n)\ =\ o(c(\log n)/\log n)\quad\text{as }n\to\infty. $$

In order to bound $\|\wh{L}^{*}_{n-I_{n}'}-\wh{L}^{*}_{n\eta'}\|_{1}$, we estimate 
\begin{align*}
&\left\|\wh{L}^{*}_{n-I_{n}'}-\wh{L}^{*}_{n\eta'}\right\|_{1}\ =\ \left\|\wh\Lambda^{*}_{\log (n-I_{n}')}-\wh\Lambda^{*}_{\log(n\eta')}\right\|_{1}\\
&=\ \left\|\sum_{k=0}^{\infty}\1_{\{\log (n-I_{n}')\wedge \log(n\eta')< -\log \wh\Pi^{*}_{k}\le\log (n-I_{n}')\vee \log(n\eta')\}}\right\|_{1}\\
&=\ \left\|\sum_{k=0}^{\infty}\1_{\{\log (n-I_{n}')\wedge \log^{+}(n\eta')< -\log \wh\Pi^{*}_{k}\le\log (n-I_{n}')\vee \log^{+}(n\eta')\}}\right\|_{1}\\
&\le\ \sum_{k=0}^{\infty}\Prob\big(\log (n-I_{n}')\wedge \log^{+}(n\eta')< -\log \wh\Pi^{*}_{k}\le\log (n-I_{n}')\vee \log^{+}(n\eta')\big)\\
&=\ \Erw\wh\Lambda_{\log (n-I_{n}')\vee\log^{+}(n\eta')}^{*}\ -\ \Erw\wh\Lambda_{\log (n-I_{n}')\wedge \log^{+}(n\eta')}^{*}.
\intertext{Now use $\Erw \wh\Lambda^{*}_{t}=\mu_{0}^{-1}t^{+}$ to see that the last line can be further estimated by}
&= \frac{1}{\mu_{0}}\Big((\Erw (\log (n-I_{n}')\vee \log^{+}(n\eta')) - \Erw (\log (n-I_{n}')\wedge \log^{+}(n\eta'))\Big)\\
&=\ \frac{1}{\mu_{0}}\Erw\left|\log (n-I_{n}')-\log^{+}(n\eta')\right|\\
&\le\ \frac{1}{\mu_{0}}\Erw\left|\log \Big(1-\frac{I_{n}'}{n}\Big)-\log\eta'\right|+\frac{1}{\mu_{0}}\Erw\left|\log(n\eta')-\log^{+}(n\eta')\right|\\
&=\ \frac{1}{\mu_{0}}d_{1}\left(\log \Big(1-\frac{I_{n}}{n}\Big),\log\eta\right)+\frac{1}{\mu_{0}}\Erw\left((-\log(\eta n))\1_{\{n\eta\le 1\}}\right).
\end{align*} 
The first summand is $o(c(\log n)/(\log n))$, $n\to\infty$, by Condition \eqref{multiplicative_finite_variance_condition} in part (a) and by Condition
\eqref{mult_sv_variance_condition} in (b) and (c). As for the second summand, an integration by parts yields
\begin{align*}
\frac{1}{\mu_{0}}\Erw((-\log(\eta n))\1_{\{n\eta\le 1\}})\ &=\ \frac{1}{\mu_{0}}\int_{\log n}^{\infty}(s-\log n)\ \Prob(-\log\eta\in ds)\\
&=\ \frac{1}{\mu_{0}}\int_{\log n}^{\infty}\Prob(-\log\eta>s)\ ds\\
&=\ 1-r(\log n),
\end{align*}
which is of the order $o(c(\log n)/(\log n))$ as $n\to\infty$ by $\eqref{r_vs_c}$ and subsequent remarks. This completes the proof.\qed

\section{Appendix}

\subsection{Convergence of renewal quantities in minimal $L^{p}$-distance}

\begin{proposition}\label{renewal_{p}rocess_d_{p}_conv}
Let $\xi,\xi_{1},\xi_{2},...$ be iid positive and nonarithmetic random variables with finite mean $\mu$ and associated zero-delayed random walk $(S_{n})_{n\ge 0}$. For $t\ge 0$, let
$$ N_{t}\ :=\ \sum_{n\ge 0}\1_{\{S_n\leq t\}} $$
denote the number of renewals in $[0,t]$.
\begin{description}[(R3)]\itemsep2pt
\item[(R1)] If $\sigma^{2}:={\rm Var}\,\xi<\infty$, then
\begin{equation*}
d_{2}\left(\frac{N_{t}-\mu^{-1}t}{\sigma\mu^{-3/2}t^{1/2}},\,\cN(0,1)\right)\ \stackrel{t\to\infty}{\longrightarrow}\ 0.
\end{equation*}
\item[(R2)] If $\sigma^{2}=\infty$,
\begin{equation*}
\ell(t)\ :=\ \Erw [\xi^{2} 1_{\{\xi\le t\}}]
\end{equation*}
is slowly varying at infinity and $c(t)$ a positive continuous function such that
\begin{equation*}
\lim_{t\to\infty}\frac{t\ell(c(t))}{c^{2}(t)}=1,
\end{equation*}
then
\begin{equation*}
d_{1}\left(\frac{N_{t}-\mu^{-1}t}{\mu^{-3/2}c(t)},\,\cN(0,1)\right)\ \stackrel{t\to\infty}{\longrightarrow\ 0}.
\end{equation*}
\item[(R3)] If, for some $\alpha\in(1,2)$, 
$$ \ell(t)\ :=\ t^{\alpha}\,\Prob(\xi>t) $$
is slowly varying at infinity and $c(t)$ a positive function satisfying
\begin{equation*}
\lim_{t\to\infty}\frac{t\ell(c(t))}{c^{\alpha}(t)}=1,
\end{equation*}
then
\begin{equation*}
d_{1}\left(\frac{N_{t}-\mu^{-1}t}{\mu^{-(1+\alpha)/\alpha}c(t)},\cS_{\alpha}\right)\ \stackrel{t\to\infty}{\longrightarrow}\ 0,
\end{equation*}
where $\cS_{\alpha}$ denotes the $\alpha$-stable law with characteristic function \eqref{cf_stable_law}.
\end{description}
All assertions remain valid with $t=nd$ if $\xi$ is $d$-artihmetic for some $d>0$.
\end{proposition}

\begin{proof}
Replacing the convergence in $d_{p}$ for $p=1$ or $2$ with weak convergence, the proposition is well-known (see, for instance, \cite[Chapter III, Section 5]{Gut:09}). Since $\sigma^{2}<\infty$ further implies
$$ \lim_{t\to\infty}\,\Erw\left(\frac{N_{t}-\mu^{-1}t}{\sqrt{\sigma^{2}\mu^{-3}t}}\right)^{2}\ =\ \Erw Z^{2}\ =\ 1,\quad Z\eqdist\cN(0,1) $$
(see e.g. \cite[Theorem 8.4 in Chapter III]{Gut:09}), the assertion in (R1) follows by an appeal to (P5) of Proposition \ref{was_dis_{p}rop}. But (R2) and (R3) follow in the same manner when invoking Lemma A.1 in \cite{IksMarMei:14} which states the convergence of the first absolute moment of $(N_{t}-\mu^{-1}t)/c(t)$ as $t\to\infty$ to the first absolute moment of the respective limiting distribution.\qed
\end{proof}

\begin{remark}\label{c_t_existence_and_properties}
The function $c(t)$ appearing in the cases (R2) and (R3) always exists provided that $\Prob(\xi>t)$ varies regularly at infinity\footnote{In the case (R2) slow variation of $\ell(t)$ is equivalent to regular variation of $\Prob(\xi>t)$ with index $-2$.}, i.e. $\Prob(\xi>t)=t^{-\alpha}\ell(t)$ for some slowly varying function $\ell$ and $\alpha\in(1,2]$. This function can be equivalently defined, see e.g. \cite[Theorem 7]{Feller:49}, by the asymptotic relation $\Prob(\xi>c(t))\simeq 1/t$, as $t\to\infty$, in particular one can choose
$$ c(t):=\Big(1/\Prob(\xi>t))\Big)^{\leftarrow}. $$
Moreover, by Proposition 1.5.15 in \cite{BingGolTeug:89}
$$ c(t)\simeq t^{1/\alpha}(\ell^{\#}(t))^{1/\alpha}\quad\text{as }t\to\infty, $$
where $\ell^{\#}$ denotes the de Bruijn conjugate of the slowly varying function $L(t): = 1/\ell (t^{1/\alpha})$. Since $\ell^{\#}$ is slowly varying as well, $c$ is regularly varying of index $1/\alpha$, a fact which has repeatedly been used in our proofs.
\end{remark}

\subsection{A linear recursion}

Fixing $a\in\N_{0}$ and a sequence $(r_{n})_{n\ge 1}$ of positive reals, let $(p_{n,k})_{0\le k<n}$ be an arbitrary probability distribution on $\{0,\ldots,n-1\}$ for each $n>a$ and then $(s_{n})_{n\ge 0}$ the unique solution to the linear recursion
\begin{equation}\label{recursion}
s_{n}\ =\ r_{n}+\sum_{k=0}^{n-1} p_{n,k}s_{k},\quad n>a,
\end{equation}
with given initial values $s_{0},s_{1},\ldots,s_{a}$. The following result forms a slight extension of Lemma 6.1 from \cite{GneIksMarMoh:14}.

\begin{lemma}\label{L61} \label{boundedness}
Suppose there exists a sequence $(\psi_{n})_{n\ge 1}$ such that
\begin{description}[(C2)]\itemsep2pt
\item[\rm(C1)] $\liminf_{n\to\infty}n^{-1}\psi_{n}\sum_{k=0}^{n-1}(n-1-k)p_{n,k}>0$,
\item[\rm(C2)] the sequence $(r_{k}\psi_{k}/k)_{k\ge 1}$ is nonincreasing.
\end{description}
Then $(s_{n})_{n\ge 0}$, defined by \eqref{recursion}, satisfies
\begin{equation}\label{bounded}
s_{n}\ =\ O\left(\sum_{k=1}^{n}\frac{r_{k}\psi_{k}}{k}\right)\quad\text{as }n\to\infty.
\end{equation}
\end{lemma}

\begin{proof}
By assumption (C1), there exists $n_{0}>a$ such that
$$ c^{-1}\ :=\ \inf_{n>n_{0}}\frac{\psi_{n}}{n}\sum_{k=0}^{n-1}p_{n,k}(n-1-k)\ >\ 0, $$
thus $c<\infty$.  We show by induction that 
$$ s_{n}\ \le\ \sum_{i=0}^{n_{0}}s_{i}\ +\ c\sum_{i=1}^{n}\frac{r_{i}\psi_{i}}{i}\quad\text{for all }n\in\N_{0}. $$
As this estimate obviously holds for $n\le n_{0}$, fix $n>n_{0}$ and suppose that it be true for all $k<n$. Defining $M^{*}:=\sum_{i=0}^{n_{0}}s_{i}$, we find
\begin{align*}
s_{n}\ &=\ r_{n}\ +\ \sum_{k=0}^{n-1}p_{n,k}s_{k}\\
&\le\ r_{n}\ +\ \sum_{k=0}^{n-1} p_{n,k}\left(M^{*}\ +\ c\sum_{i=1}^{k}\frac{r_{i}\psi_{i}}{i}\right)\\
&=\ M^{*}\ +\ r_{n}\ +\ c\sum_{k=1}^{n-1}p_{n,k}\sum_{i=1}^{k}\frac{r_{i}\psi_{i}}{i}\\
&=\ M^{*}\ +\ r_{n}\ +\ c\sum_{i=1}^{n-1} \frac{r_{i}\psi_{i}}{i}\sum_{k=i}^{n-1}p_{n,k}\\
&=\ M^{*}\ +\ r_{n}\ +\ c\sum_{i=1}^{n-1} \frac{r_{i}\psi_{i}}{i}\left(1-\sum_{k=0}^{i-1}p_{n,k}\right)\\
&\le\ M^{*}\ +\ r_{n}\ +\ c\sum_{i=1}^{n-1}\frac{r_{i}\psi_{i}}{i}\ -\ c\,\frac{r_{n}\psi_{n}}{n}\sum_{i=1}^{n-1} \sum_{k=0}^{i-1}p_{n,k}\\
&\le\ M^{*}\ +\ r_{n}\ +\ c\sum_{i=1}^{n-1} \frac{r_{i}\psi_{i}}{i}\ -\ c\,\frac{r_{n}\psi_{n}}{n}\sum_{k=0}^{n-1}p_{n,k}(n-1- k)\\
&\le\ M^{*}\ +\ c\sum_{i=1}^{n-1}\frac{r_{i}\psi_{i}}{i},
\end{align*}
where the last inequality follows from the definition of $c$.\qed
\end{proof}

We actually need the following generalization of the previous result.

\begin{lemma}\label{L61gen}
Suppose that $(s_{n})_{n\ge 0}$ satisfies
\begin{equation}\label{recursion_ineq}
s_{n}\ \le\ r_{n}\ +\ \sum_{k=0}^{n-1}p_{n,k}s_{k}\quad\text{for }n>a
\end{equation}
for a nonnegative sequence $(r_{n})_{n\ge 1}$ and initial values $s_{0},...,s_{a}$. Suppose further the existence of a sequence $(\psi_{n})_{n\ge 1}$ such that (C1) holds and
\begin{description}[(C2)]
\item[(C3)] the sequence $(r_{k}\psi_{k}/k)_{k\ge 1}$ is bounded. 
\end{description}
Then $(s_{n})_{n\ge 0}$ satisfies
\begin{equation}\label{bounded gen}
s_{n}\ =\ O\left(\sum_{k=1}^{n}\frac{r_{k}^{*}\psi_{k}}{k}\right)\quad\text{as }n\to\infty,
\end{equation}
where $\displaystyle r_{k}^{*}=\frac{k}{\psi_{k}}\,\sup_{j\ge k}\frac{r_{j}\psi_{j}}{j}$ for $k\ge 1$.
\end{lemma}

\begin{proof}
First note that $(\psi_{n})_{n\ge 1}$ satisfies (C2) of the previous lemma with $(r_{n}^{*})_{n\ge 1}$ instead of $(r_{n})_{n\ge 1}$. Let $(s_{n}^{*})_{n\ge 0}$ be defined by $s_{k}^{*}:=s_{k}$ for $k=0,...,a$ and $s_{n}^{*}=r_{n}^{*}+\sum_{k=0}^{n-1}p_{n,k}s_{k}^{*}$ for $n>a$. Since $r_{n}^{*}\ge r_{n}$ for each $n$ and
$$ s_{n}^{*}-s_{n}\ \ge\ r_{n}^{*}-r_{n}\ +\ \sum_{k=0}^{n-1}p_{n,k}(s_{k}^{*}-s_{k})\quad\text{for }n>a, $$
a simple induction shows that $s_{n}\le s_{n}^{*}$ for all $n\ge 0$. Now use Lemma \ref{L61} to infer
\begin{equation*}
s_{n}^{*}\ =\ O\left(\sum_{k=1}^{n}\frac{r_{k}^{*}\psi_{k}}{k}\right)\quad\text{as }n\to\infty,
\end{equation*}
which in combination with the previous statement proves \eqref{bounded gen}.\qed
\end{proof}

\vspace{1cm}
\footnotesize
\noindent   {\bf Acknowledgements.}
The research of Gerold Alsmeyer was supported by the Deutsche Forschungsgemeinschaft (SFB 878), the research of Alexander Marynych by the Alexander von Humboldt Foundation.

\bibliographystyle{abbrv}
\bibliography{StoPro}

\end{document}